\documentclass[12pt]{Paper}

\usepackage{scalerel}
\usepackage{amsthm,lipsum}
\usepackage{amsmath}
\usepackage{amssymb}
  \usepackage{paralist}
  \usepackage{graphics} 
  \usepackage{epsfig} 
  \usepackage{csquotes}
  \usepackage{subcaption}
\usepackage{graphicx}  \usepackage{epstopdf}
 \usepackage[colorlinks=true]{hyperref}
\usepackage{bm}
\usepackage{siunitx}
\hypersetup{urlcolor=blue, citecolor=red}

\usepackage[margin=1.2in]{geometry}

\newtheorem{theorem}{Theorem}[section]

\newtheorem{proposition}[theorem]{Proposition}

\newtheorem{innernclaim}{Claim}
\newenvironment{nclaim}[1]
  {\renewcommand\theinnernclaim{#1}\innernclaim}
  {\endinnernclaim}

\theoremstyle{definition}
\newtheorem{remark}{Remark}[section]
\newtheorem{example}{Example}[section]

\usepackage[T1]{fontenc}    
\usepackage[utf8]{inputenc} 

\def\f{\longrightarrow}

\def\e{\varepsilon}

\def\x{\bar{x}}

\def\<{\langle}
\def\>{\rangle}

\def\L{\mathcal{L}}

\def\e{\varepsilon}
\def\R{\mathbb{R}}

\def\inte{\textnormal{int}\,}
\def\clo{\textnormal{cl}\,}

\def\bdry{\textnormal{bdry}\,}

\def\proj{\textnormal{proj}\,}

\def\nvarphi{\bp\nvarphi}
\def\sp{\hspace{0.015cm}}
\def\bp{\hspace{-0.08cm}}
\def\bbp{\hspace{-0.05cm}}

\allowdisplaybreaks

\title[The Extended Exterior Sphere Condition]{The Extended Exterior Sphere Condition}

\author[C. Nour and J. Takche]{}


\thanks{$^*$ Corresponding author}

\begin{document}

\maketitle

\centerline{\scshape Chadi Nour$^*$}
\medskip
{\footnotesize
\centerline{Department of Computer Science and Mathematics}
\centerline{Lebanese American University}
   \centerline{Byblos Campus, P.O. Box 36}
   \centerline{Byblos, Lebanon}
   \centerline{E-mail: cnour@lau.edu.lb}
}

\medskip

\centerline{\scshape Jean Takche}
\medskip
{\footnotesize

\centerline{Department of Computer Science and Mathematics}
\centerline{Lebanese American University}
   \centerline{Byblos Campus, P.O. Box 36}
   \centerline{Byblos, Lebanon}
   \centerline{E-mail: jtakchi@lau.edu.lb}
}

\begin{abstract} We prove that the complement of a closed set $S$ satisfying an {\it extended} exterior sphere condition is nothing but the union of closed balls with {\it common} radius. This generalizes \cite[Theorem 3]{Nacry} where the set $S$ is assumed to be prox-regular, a property {\it stronger} than the extended exterior sphere condition. We also provide a sufficient condition for the equivalence between prox-regularity and the extended exterior sphere condition that generalizes \cite[Corollary 3.12]{NT1} to the case in which $S$ is not necessarily regular closed.
\end{abstract}

\section{Introduction} Let $S\subset\R^n$ be a closed set. For $x\in S$, a vector $\zeta \in \R^n$ is said to be {\it proximal normal} to $S$  at $x$ if there exists $\sigma = \sigma(x,\zeta) \geq 0$ such that 
\begin{equation} \label{psi}
\langle \zeta,s-x\rangle \leq \sigma \|s-x\|^2,\;\;\;\forall s\in S,
\end{equation}
where $\<\cdot,\cdot\>$ and $\|\cdot\|$ denote the standard inner product and Euclidean norm, respectively.
The relation (\ref{psi}) is commonly referred to as the {\em proximal normal inequality}. Note that for $\zeta\not=0$ and $\sigma\not=0$, the proximal normal inequality (\ref{psi}) is equivalent to \begin{equation*}\label{proxball1} B\bp\left(x+\frac{1}{2\sigma}\frac{\zeta}{\|\zeta\|};{\frac{1}{2\sigma}}\right)\cap S=\emptyset,\end{equation*}
where $B(y;\rho)$ denotes the open ball of radius $\rho$ centered at $y$. In that case, we say that $\zeta$ is {\it realized} by a $\frac{1}{2\sigma}$-sphere.  On the other hand, for $\zeta\not=0$ and $\sigma=0$, the proximal normal inequality (\ref{psi}) is equivalent to $B(x+\rho\zeta;\rho)\cap S=\emptyset$ for all $\rho>0.$ In that case, $\zeta$ is realized by an $\rho$-sphere for any $\rho>0$. Now, in view of (\ref{psi}), the set of all proximal normals to $S$ at $x$ is a convex cone containing $0$, and we denote it by $N^P_S(x)$. Since no nonzero $\zeta$ satisfying \eqref{psi} exists if $x\in\inte S$ (the interior of $S$), we deduce that $N^P_S(x)=\{0\}$ for all $x\in\inte S$. This may also occur for $x\in\bdry S$, the boundary of $S$, as is the case when $S$ is the epigraph of the function $f(z) = -|z|$ and $x$ is the origin. More information about $N^P_S(\cdot)$ and proximal analysis can be found in \cite{clsw,mord,penot,rockwet}.

Now we fix $r>0$. We recall that the set $S$ is said to be {\it $r$-prox-regular} if for any $x\in\bdry S$ and {\it for any} $0\not=\zeta\in N^P_S(x)$, $\zeta$ is realized by an $r$-sphere. For more information about prox-regularity, and related properties such as {\it positive reach}, {\it proximal smoothness}, {\it $p$-convexity} and $\varphi_0$-{\it convexity}, see \cite{canino,csw,cm,fed,prt,shapiro}. On the other hand, the set $S$ is said to be satisfying the {\it exterior $r$-sphere condition} if for any $x\in\bdry S$, {\it there exists} $0\not=\zeta\in N^P_S(x)$ such that $\zeta$ is realized by an $r$-sphere. Clearly if $S$ is $r$-prox-regular, then it satisfies the exterior $r$-sphere condition. For the converse, it does not hold in general as it is shown via counterexamples in \cite{NT1}. In this latter, Nour, Stern and Takche provided sufficient condition for the equivalence between prox-regularity and the exterior sphere condition. More precisely, they proved that the two properties are equivalent  if $S$ is {\it epi-Lipschitz} (or {\it wedged}\,) and has compact boundary, see \cite[Corollary 3.12]{NT1}.\footnote{This equivalence result is generalized to the variable radius case in \cite{nst1}, see also \cite{nst5bis}.} Recall that a closed set $S$ is said to be epi-Lipschitz at a point $x\in S$ if the set $S$ can be viewed near $x$, after application of an orthogonal matrix, as the epigraph of a Lipschitz continuous function. If this holds for all $x\in S$, then we simply say that $S$ is epi-Lipschitz. This geometric definition was introduced by Rockafellar in \cite{rock}. The epi-Lipschitz property of $S$ at $x$ is also characterizable in terms of the nonemptiness of the topological interior of the Clarke tangent cone of $S$ at $x$ which is also equivalent to the {\it pointedness} of the Clarke normal cone of $S$ at $x$, see \cite[Theorem 7.3.1]{clarkeold}.

In their recent paper \cite{Nacry}, Nacry and Thibault proved in \cite[Theorem 3]{Nacry} that if $S$ is $r$-prox-regular, then the complement of $S$, denoted by $S^c$, is the union of closed balls with {\it common} radius. What is remarkable in this result is that the set $S$ can be any closed set, including the sets that are not {\it regular closed} (that is, $S\not=\clo(\inte S)$  the closure of the interior of $S$).

In this paper, we generalize  \cite[Theorem 3]{Nacry} by replacing the $r$-prox-regulariy of $S$ by {\it weaker} conditions, including the exterior $r$-sphere condition if $S$ is regular closed, and a new {\it extended} version of the exterior $r$-sphere condition if $S$ is not necessarily regular closed.

When $S$ is regular closed, the exterior $r$-sphere condition of $S$ coincides with the {\it interior $r$-sphere condition} of $(\inte S)^c$, a well known condition in control theory, see e.g., \cite{cf,cs1,cs2,NS2008}.\footnote{ In these references, the interior sphere condition is used to study the regularity of the unilateral and the bilateral minimal time functions associated to a control system.} Using this fact and \cite[Conjecture 1.2]{nst5},\footnote{This conjecture is introduced in \cite{NT1,nst2}, and proved in \cite{nst5}. A generalization of this result to the variable radius case is given in  \cite[Theorem 3.1]{nst5}, see also \cite{nourtakchejota}.} we obtain the following theorem which is the first main result of this paper.

\begin{theorem}\label{th1} Let $S\subset \R^n$ be a closed set such that $\clo(\inte S)$ satisfies the exterior $r$-sphere condition for some $r>0$. Then $(\inte S)^c$ is the union of closed $\frac{r}{2}$-balls. If in addition $S$ is regular closed, then $S^c$ is the union of closed $r'$-balls for any $r'<\frac{r}{2}$.
\end{theorem} 
Note that the \enquote{In addition} part of Theorem \ref{th1} is not necessarily true  when $S$ is not regular closed, see Example \ref{ex1}. On the other hand, the assumption \enquote{$\clo(\inte S)$ satisfies the exterior $r$-sphere condition} is {\it weaker} than the assumption \enquote{$S$ satisfies the exterior $r$-sphere condition}, see Proposition \ref{prop1} and Example \ref{ex2}.

In the following theorem, which is our second main result, we prove that a similar result to \cite[Theorem 3]{Nacry} can be obtained if the prox-regularity is replaced by a weaker property, namely the following {\it extended} exterior $r$-sphere condition. We say that $S$ satisfies the  extended exterior $r$-sphere condition if:
\begin{itemize}
\item  For any $x\in\bdry(\inte S)\subset \bdry S$, there exists $0\not=\zeta\in N^P_S(x)$ such that $\zeta$ is realized by an $r$-sphere. 
\item For any $x\in(\bdry S)\setminus \bdry(\inte S)$ and for any $0\not=\zeta\in N^P_S(x)$, $\zeta$ is realized by an $r$-sphere. 
\end{itemize}

\begin{theorem}\label{th2} Let $S\subset \R^n$ be a closed set  satisfying the extended exterior $r$-sphere condition for some $r>0$. Then  $S^c$ is the union of closed $\frac{r}{2}$-balls. \end{theorem}

One can easily see that if $S$ satisfies the extended exterior $r$-sphere condition then $S$ satisfies the exterior $r$-sphere condition. The converse, which does not hold in general (see Example \ref{ex2}), is valid  if  $S$ is regular closed. Note that the extended exterior $r$-sphere condition of Theorem \ref{th2} cannot be replaced by the exterior $r$-sphere condition as it will be shown in Example \ref{ex3}.  On the other hand, clearly if $S$ is $r$-prox-regular then $S$ satisfies the extended exterior $r$-sphere condition. The converse does not hold in general, even if $S$ is regular closed, as it is proved in the counterexamples of \cite[Section 2]{NT1}. In the following theorem which is our third and last main result, we provide a sufficient condition for the equivalence between prox-regularity and the extended exterior sphere condition. It generalizes \cite[Corollary 3.12]{NT1} to the case in which $S$ is not necessarily regular closed. 

\begin{theorem}\label{th3} Let $S\subset\R^n$ be a closed set such that $\clo(\inte S)$ is epi-Lipschitz and has compact boundary. Then $S$ is $r$-prox-regular for some $r>0$ if and only if $S$ satisfies the extended exterior $r'$-sphere condition for some $r'>0$.

\end{theorem}

The details of the proofs of Theorems \ref{th1}, \ref{th2} and \ref{th3} will be presented in the next section after giving a brief description of the notation used in this paper, and providing some auxiliary results.

 \section{Proofs of the main results} We denote by $\|\cdot\|$, $\<\cdot,\cdot\>$, $B$ and $\bar{B}$, the Euclidean norm, the usual inner product, the open unit ball and the closed unit ball, respectively. For $\rho>0$ and $x\in\R^n$, we set $B(x;\rho):= x + \rho B$ and $\bar{B}(x;\rho):= x + \rho \bar{B}$. For a set $S\subset\R^n$, $S^c$, int\,$S$, $\textnormal{bdry}\,S$ and $\textnormal{cl}\,S$ are the complement (with respect to $\R^n$), the interior, the boundary and the closure of $S$, respectively.  The closed segment joining two points $x$ and $y$ in $\R^n$ is denoted by $[x,y]$. The distance from a point $x$ to a set $S$ is denoted by $d_S(x)$. We also denote by $\textnormal{proj}\,_S(x)$ the set of closest points in $S$ to $x$, that is, the set of points $s\in S$ which satisfy $d_A(x)=\|s-x\|$. For $A$ and $B$ two subsets of $\R^n$, $d(A,B)$ denotes the distance between $A$ and $B$, that is, $$d(A,B):=\inf\{\|a-b\| : (a,b)\in A\times B\}.$$
 
 For $S\subset \R^n$  closed, the following proposition provides some useful properties of the set $\clo (\inte S)$.
 
 \begin{proposition}\label{prop1} Let $S\subset \R^n$ be closed. Then we have the following$\sp:$
 \begin{enumerate}[$(i)$]
 \item $\clo (\inte S)\subset S$, $\inte(\clo (\inte S))=\inte S$, and $$\bdry(\clo (\inte S))=\bdry(\inte S)\subset \bdry S.$$
 \item $S=[\clo(\inte S)]\cup [(\bdry S)\setminus \bdry(\inte S)]$ with $$[\clo(\inte S)]\cap [(\bdry S)\setminus \bdry(\inte S)]=\emptyset.$$
  \item If $S=\clo O$, where $O$ is open, then $S$ is regular closed.
 \item If for $r>0$, $S$ satisfies the exterior $r$-sphere condition, then $\clo (\inte S)$ also  satisfies the exterior $r$-sphere condition.
 \end{enumerate}
 \end{proposition}
 \begin{proof} The properties $(i)$-$(iii)$ are well known in metric topology. For $(iv)$, let $r>0$ and assume that $S$ satisfies the exterior $r$-sphere condition. Let $x\in\bdry (\clo (\inte S))$. By $(i)$ we have that $x\in \bdry S$, and hence there exists $0\not=\zeta\in N_S^P(x)$ such that $\zeta$ is realized by an $r$-sphere. This gives using the proximal normal inequality that $$\left\<\frac{\zeta}{\|\zeta\|},s-x\right\>\leq \frac{1}{2r}\|s-x\|^2,\;\;\forall s\in S.$$
 Since $\clo (\inte S)\subset S$, we obtain that $$\left\<\frac{\zeta}{\|\zeta\|},s-x\right\>\leq \frac{1}{2r}\|s-x\|^2,\;\;\forall s\in\clo (\inte S).$$
 Therefore, $\zeta\in N^P_{\clo (\inte S)}(x)$ and is realized by an $r$-sphere. \end{proof}
 
 \begin{example}\label{ex2} The converse of Proposition \ref{prop1}$(iv)$ is not necessarily true. Indeed, consider in $\R^2$, $S:=\bar{B}(0;1)\cup [(1,0),(2,0)]$. Clearly $S$ does not satisfy the exterior $r$-sphere condition for any $r>0$, but $\clo (\inte S)=\bar{B}(0;1)$ satisfies the exterior $r$-sphere condition for any $r>0$. 
 \end{example}
 
 In the following proposition we list some well known characterizations and consequences of prox-regular and epi-Lipshitz properties. For the proofs, see \cite{clarkeold,prt,rockwet}.
 
 \begin{proposition}\label{prop2}   Let $S\subset \R^n$ be closed. Then we have the following$\sp:$
 \begin{enumerate}[$(i)$]
 \item If $S$ is prox-regular, then for each $x\in\bdry S$ we have \begin{equation*}N_S^P(x)=N_S^L(x)=N_S(x), \end{equation*}
where $N_S^L(x)$ and $N_S(x)$ denotes the \textnormal{Mordukhovich} $($or \textnormal{limiting}$)$ and the  \textnormal{Clarke} normal cones to $S$ at $x$, respectively.
\item For $x\in\bdry S$, $S$ is epi-Lipschitz at $x$ if and only if $N_S(x)$ is pointed, that is, $N_S(x)\cap-N_S(x)=\{0\}$.
\item If for $x\in\bdry S$, $S$ is epi-Lipschitz at $x$, then $x\in\clo(\inte S)$.
  \end{enumerate}
 
 \end{proposition}
 
 \subsection{Proof of Theorem \ref{th1}}   Let $S\subset \R^n$ be a closed set such that $\clo(\inte S)$ satisfies the exterior $r$-sphere condition for some $r>0$. Since by Proposition \ref{prop1}$(iii)$ we have that $\clo(\inte S)$ is regular closed, we get that $(\inte(\clo(\inte S))^c=(\inte S)^c$ satisfies the interior $r$-sphere condition. Hence by  \cite[Conjecture 1.2]{nst5} we obtain that $(\inte S)^c$ is the union of closed $\frac{r}{2}$-balls. 
 
 We proceed to prove the \enquote{In addition} part of Theorem \ref{th1}. We assume that $S$ is regular closed, and we consider $r'<\frac{r}{2}$.  Let $x\in S^c$. Since $S^c\subset (\inte S)^c$ and $(\inte S)^c$ is the union of closed $\frac{r}{2}$-balls, there exists $y_x\in(\inte S)^c$ such that \begin{equation} \label{f1} x\in \bar{B}\left(y_x;\frac{r}{2}\right)\subset (\inte S)^c.\end{equation} 
 If $\|y_x-x\|<r'$, then using \eqref{f1} we get that $$ x\in \bar{B}(y_x;r')\subset  {B}\left(y_x;\frac{r}{2}\right)\subset\bar{B}\left(y_x;\frac{r}{2}\right) \subset (\inte S)^c.$$
This gives that $$x\in \bar{B}(y_x;r')\subset \inte((\inte S)^c)=(\clo(\inte S))^c=S^c.$$
 Now we assume that $\|y_x-x\|\geq r'$. Then by \eqref{f1} and since $x\in S^c$, we have that \begin{eqnarray*} x\in \bar{B}\left(x+r'\frac{y_x-x}{\|y_x-x\|};r'\right)& \subset & {B}\left(y_x;\|y_x-x\|\right)\cup\{x\} \\&\subset& {B}\left(y_x;\frac{r}{2}\right)\cup\{x\}\\&\subset& \inte((\inte S)^c)\cup S^c=S^c.\end{eqnarray*}
 Therefore, $S^c$ is the union of closed $r'$-balls. This terminated the proof of the  \enquote{In addition} part, and hence the proof of Theorem \ref{th1}. \hfill $\Box$
 
  In the following example, we prove that the \enquote{In addition} part of Theorem \ref{th1} is not necessarily true  when $S$ is not regular closed.
 
 \begin{example}\label{ex1} Let $S:=\{(x,e^x) : x\in\R\}\cup\{(x,-e^{x}) : x\in\R\}$. Clearly $S$ is not regular closed, and satisfies the exterior $ \frac{3\sqrt{3}}{2}$-sphere condition. On the other hand, $S^c$ is not the union of closed $r$-balls for {\it any} $r>0$. Note that $S$ does not satisfy the extended exterior $r$-sphere condition for any $r>0$.
 
 \end{example}
 
  \subsection{Proof of Theorem \ref{th2}} Let $S\subset \R^n$ be a closed set  satisfying the extended exterior $r$-sphere condition for some $r>0$. We consider $x\in S^c$. Let $s_0\in\proj_S(x)$ and let $r_0:=\|x-s_0\|$. Clearly we have $s_0\in\bdry S$ and $s_0\not=x$, and hence $r_0\not=0$. Moreover, we have \begin{equation}\label{beforecase} B(x;r_0)\cap S=\emptyset. \end{equation}
There are three cases to consider.\vspace{0.1cm} \\
\underline{Case 1}: $r_0> \frac{r}{2}$.\vspace{0.1cm} \\
Then using \eqref{beforecase}, we obtain that \begin{equation*}\label{case1} x\in \bar{B}\left(x;\frac{r}{2}\right)\subset B(x;r_0)\subset S^c. \vspace{0.15cm}\end{equation*}
\underline{Case 2}: $r_0\leq \frac{r}{2}$ and $s_0\not \in \bdry(\inte S)$. \vspace{0.1cm} \\
Then the proximal normal vector $\zeta_0:=x-s_0\in N_S^P(s_0)$ is realized by an $r$-sphere. This gives that \begin{equation}\label{case2.1} B\left(s_0+r\frac{\zeta_0}{\|\zeta_0\|};r\right)\subset S^c. \end{equation}
We claim that \begin{equation}\label{case2.2} \bar{B}\left(x+\frac{r}{2}\frac{\zeta_0}{\|\zeta_0\|};\frac{r}{2}\right)\subset B\left(s_0+r\frac{\zeta_0}{\|\zeta_0\|};r\right).\end{equation}
Indeed, for $y\in \bar{B}\left(x+\frac{r}{2}\frac{\zeta_0}{\|\zeta_0\|};\frac{r}{2}\right)$ we have 
\begin{eqnarray*} \left\|y-s_0-r\frac{\zeta_0}{\|\zeta_0\|}\right\| &=&   \left\|\left(y-x-\frac{r}{2}\frac{\zeta_0}{\|\zeta_0\|}\right)+\left(x-s_0-\frac{r}{2}\frac{\zeta_0}{\|\zeta_0\|}\right)\right\| \\ & \leq & \left\|y-x-\frac{r}{2}\frac{\zeta_0}{\|\zeta_0\|}\right\|+\left\|x-s_0-\frac{r}{2}\frac{\zeta_0}{\|\zeta_0\|}\right\| \\&\leq& \frac{r}{2} + \left\|r_0\frac{\zeta_0}{\|\zeta_0\|}-\frac{r}{2}\frac{\zeta_0}{\|\zeta_0\|}\right\| \\ &=& \frac{r}{2} + \frac{r}{2}-r_0 = r-r_0 < r.\end{eqnarray*}
Now combining \eqref{case2.1} and \eqref{case2.2}, we obtain that $$x\in \bar{B}\left(x+\frac{r}{2}\frac{\zeta_0}{\|\zeta_0\|};\frac{r}{2}\right)\subset  B\left(s_0+r\frac{\zeta_0}{\|\zeta_0\|};r\right)\subset S^c.\vspace{0.15cm}$$
\underline{Case 3}: $r_0\leq\frac{r}{2}$ and $s_0\in \bdry(\inte S)$.  \vspace{0.1cm} \\
Then for every $\e>0$, we have $B(s_0;\e)\cap \inte S\not=\emptyset.$ Let $z_{\e}\in B(s_0;\e)\cap \inte S$. We denote by $\xi_{\e}:=\frac{z_\e-x}{\|z_\e-x\|}$ and we consider $s_{\e}\in [z_\e,x]\cap \bdry S$, where this latter intersection is nonempty since $z_\e\in\inte S$ and $x\in S^c$. Let $\zeta_\e\in N_S^P(s_\e)$ be the unit vector realized by an $r$-sphere. Hence, for $y_\e:=s_\e+r{\zeta_\e}$, we have \begin{equation}\label{case3.1} B\left(y_\e;r\right)\cap S=\emptyset.\end{equation}
If $y_\e=x$, then $$x\in\bar{B}\left(x;\frac{r}{2}\right)\subset B(y_\e;r)\subset S^c.$$
If $0<\|y_\e-x\|<r$, then for $w\in  \bar{B}\left(x+\frac{r}{2}\frac{y_\e-x}{\|y_\e-x\|};\frac{r}{2}\right)$, we have 
\begin{eqnarray*}\|w-y_\e\| &\leq& \left\|w-x-\frac{r}{2}\frac{y_\e-x}{\|y_\e-x\|}  \right\| + \left\|x+\frac{r}{2}\frac{y_\e-x}{\|y_\e-x\|}-y_\e  \right\| \\ &\leq& \frac{r}{2} + \left\|\frac{r}{2}\frac{y_\e-x}{\|y_\e-x\|}-(y_\e-x)\right\| \\ &=& \frac{r}{2}+\left|\frac{r}{2}-\|y_\e-x\| \right| < r.  \end{eqnarray*}
This gives using \eqref{case3.1} that  \begin{equation*}\label{case3.2} x\in \bar{B}\left(x+\frac{r}{2}\frac{y_\e-x}{\|y_\e-x\|};\frac{r}{2}\right)\subset B\left(y_\e;r\right)\subset S^c.\end{equation*}
\begin{figure}[t]
\centering
\includegraphics[scale=0.53]{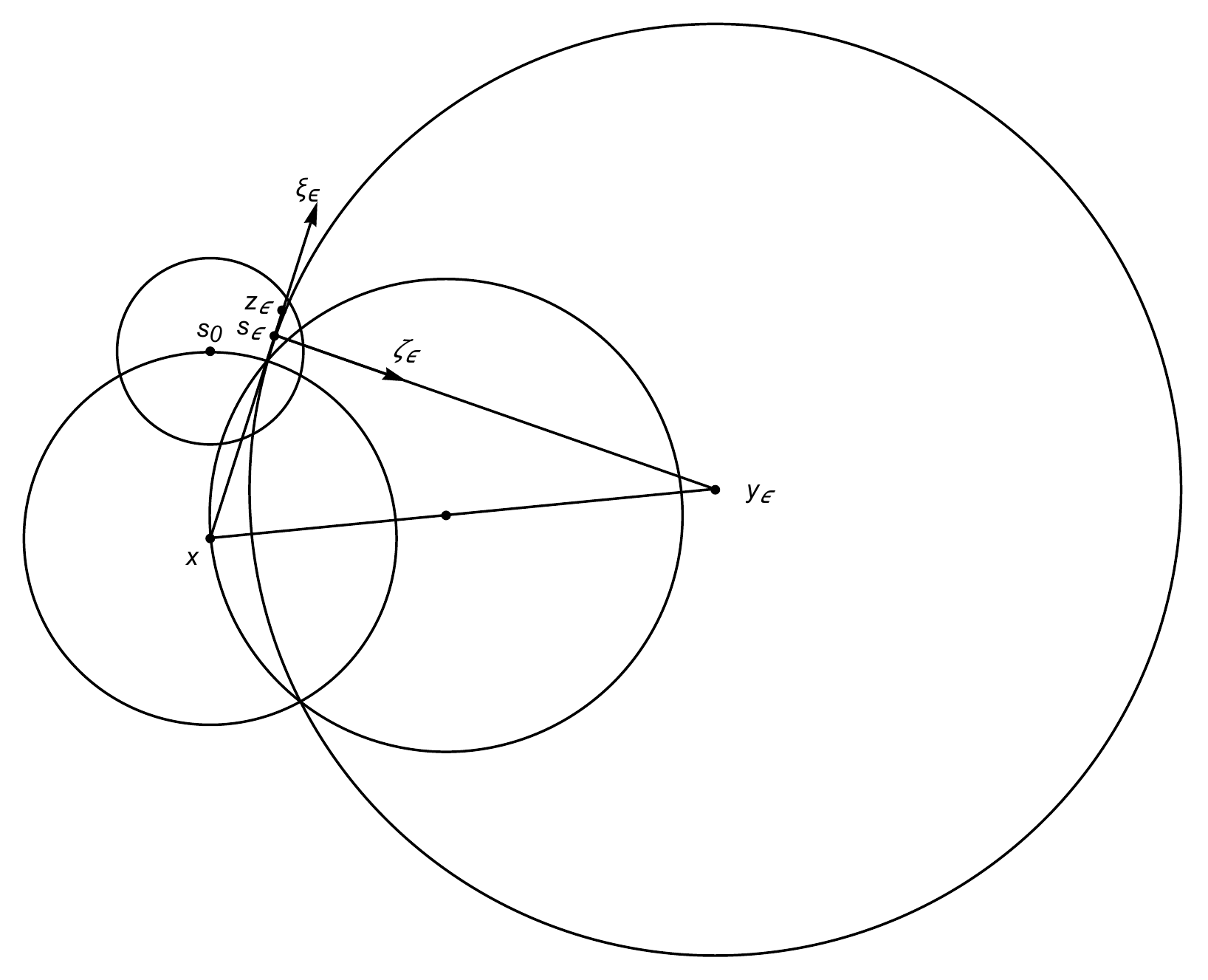}
\caption{\label{Fig1} Case 3: $\|y_\e-x \|\geq r$}
\end{figure} 
\hspace{-0.28cm} Now we assume that $\|y_\e-x\|\geq r$ and we denote by \begin{equation} \label{case3.4} r_\e:=\frac{r_0^2\|y_\e-x\|}{\|y_\e-x\|^2+r_0^2-r^2}>0.\end{equation}
\begin{nclaim}{1} \label{claim1} $B\left(x+r_\e\frac{y_\e-x}{\|y_\e-x\|}; r_\e \right) \subset B(x;r_0)\cup B(y_\e;r).$
\end{nclaim}
To prove this claim, let $w\in B\left(x+r_\e\frac{y_\e-x}{\|y_\e-x\|}; r_\e \right)$ and assume that $w\not\in B(x;r_0).$ Then \begin{equation}\label{case3.3}\|w-x\|<2r_{\e}\left\<\frac{w-x}{\|w-x\|},\frac{y_\e-x}{\|y_\e-x\|}\right\>.\end{equation}
We have \begin{eqnarray*}\|w-y_\e\|^2 &=& \|(w-x)-(y_\e-x)\|^2\\&=&  \|w-x\|^2+\|y_\e-x\|^2-2\<w-x,y_\e-x\> \\&\overset{ \eqref{case3.3}}{<}& \|w-x\|^2+\|y_\e-x\|^2 - \frac{\|y_\e-x\|}{r_\e}\|w-x\|^2\\&=& \|y_\e-x\|^2  + \|w-x\|^2\left(1-\frac{\|y_\e-x\|}{r_\e}\right)\\&\overset{\eqref{case3.4}}{=}&  \|y_\e-x\|^2  + \|w-x\|^2\left(1-\frac{\|y_\e-x\|^2+r_0^2-r^2}{r_0^2}\right) \\&=&   \|y_\e-x\|^2  + \underbrace{\|w-x\|^2}_{\geq r_0^2} \underbrace{\left(\frac{r^2-\|y_\e-x\|^2}{r_0^2}\right)}_{\leq 0}\\&\leq& \|y_\e-x\|^2 +r^2-\|y_\e-x\|^2=r^2.\end{eqnarray*}
This gives that $w\in B(y_\e,r)$ which terminates the proof of the claim.
\begin{nclaim}{2} \label{claim2}  $\left\<{\xi_\e},{\zeta_\e}\right\><\frac{\e}{2r}$.
\end{nclaim}
To prove this claim, first we remark that since $z_\e\in\inte S$, we deduce from \eqref{case3.1} that $z_\e\not\in \bar{B}(y_\e;r)=\bar{B}\left(s_\e+r{\zeta_\e};r\right)$. Hence, using the equality $z_\e-s_\e=\|z_\e-s_\e\|{\xi_\e}$, we obtain that  $$\|z_\e-s_\e\|>2r\left\<{\xi_\e},{\zeta_\e}\right\>. $$
Therefore, using that $s_0\in\proj_S(x)$, we get that  \begin{eqnarray*}\left\<{\xi_\e},{\zeta_\e}\right\>< \frac{1}{2r}\|z_\e-s_\e\|&=& \frac{1}{2r}\left(\|z_\e-x\|-\|s_\e-x\|\right)\\&\leq&  \frac{1}{2r}\left(\|z_\e-s_0\|+\underbrace{\|s_0-x\|-\|s_\e-x\|}_{\leq0}\right)\\&\leq& \frac{1}{2r}\|z_\e-s_0\|<\frac{\e}{2r}. \end{eqnarray*}
The proof of the claim is terminated.
\begin{nclaim}{3} \label{claim3} If $\e\in\left(0,\frac{r_0^3}{4r^2}\right]$ then $r_\e> \frac{r}{2}$.
\end{nclaim}
To prove the claim, first we remark using \eqref{case3.4} that $r_\e> \frac{r}{2}$ is equivalent to \begin{equation} \label{case3.5} P(\|y_\e-x\|):=r\|y_\e-x\|^2-2r^2_0\|y_\e-x\|+ r(r_0^2-r^2)<0.\end{equation}
Since the discriminant $\Delta'$ of $P(\|y_\e-x\|)$ is $(r^2-r_0^2)^2+r^2r_0^2>0$ and the product of the roots is $r_0^2-r^2<0$,  we conclude that \eqref{case3.5}  is equivalent to \begin{equation} \label{case3.7}  \|y_\e-x\|<\frac{r_0^2+\sqrt{\Delta'}}{r}. \end{equation}
We have \begin{eqnarray}  \nonumber \|y_\e-x\|^2 &=&  \|y_\e-s_\e\|^2+  \|s_\e-x\|^2+2\<y_\e-s_\e,s_\e-x\> \\ &=& \nonumber r^2+  \|s_\e-x\|^2 +2r\|s_\e-x\|\left\<{\zeta_\e},{\xi_\e}\right\>\\&\overset{\textnormal{Claim}\,\ref{claim2}}{<} & \nonumber r^2 + \|s_\e-x\|^2 +\e\|s_\e-x\| \\ &\overset{s_\e\in[z_\e,x]}{\leq}& \nonumber r^2 + \|z_\e-x\|^2+\e\|z_\e-x\| \\ &\leq & \nonumber   r^2 + (\|z_\e-s_0\|+\|s_0-x\|)^2 +\e(\|z_\e-s_0\|+\|s_0-x\|) \\ &\leq&\nonumber r^2+(\e+r_0)^2+\e(\e+r_0) \\ &=& \nonumber r^2 + r_0^2+2\e^2+3\e r_0 \\ &\overset{\e\in \left(0,\frac{r_0^3}{4r^2}\right]}{\leq}&\nonumber r^2 + r_0^2 + \frac{r_0^6}{8r^4}+\frac{3r_0^4}{4r^2}\\ &=&\nonumber r^2 + r_0^2 +\frac{r_0^4}{r^2}\left(\frac{r_0^2}{8r^2}+\frac{3}{4}\right)\\ &\overset{r_0\leq\frac{r}{2}}{<}& r^2+r_0^2+ \frac{r_0^4}{r^2}. \label{case3.6} \end{eqnarray}
On the other hand, we have 
 \begin{eqnarray*}(r_0^2+\sqrt{\Delta'})^2 &=&   2r_0^4+r^4-r^2r_0^2+2r_0^2\sqrt{\left(r^2-\frac{1}{2}r_0^2\right)^2+\frac{3r_0^4}{4}}  \\ &>&  2r_0^4+r^4-r^2r_0^2 + 2r_0^2 \left(r^2-\frac{1}{2}r_0^2\right)\\  &=& r_0^4+r^4+r^2r_0^2.\end{eqnarray*}
This gives that $$\left(\frac{r_0^2+\sqrt{\Delta'} }{r}\right)^2 > r^2+r_0^2+ \frac{r_0^4}{r^2}.$$
Combining this latter with \eqref{case3.6}, we obtain inequality \eqref{case3.7}, and this terminates the proof of Claim \ref{claim3}.

Now we fix $\e\in\left(0,\frac{r_0^3}{4r^2}\right]$. By Claim \ref{claim3} we have that $\frac{r}{2}<r_\e$. Hence, using Claim \ref{claim1}, \eqref{beforecase} and \eqref{case3.1}, we get that \begin{eqnarray*} x\in\bar{B}\left(x+\frac{r}{2}\frac{y_\e-x}{\|y_\e-x\|}; \frac{r}{2}\right) &\subset&   {B}\left(x+r_\e\frac{y_\e-x}{\|y_\e-x\|}; r_\e\right)\cup\{x\}\\ &\subset&  B(x;r_0)\cup B(y_\e;r)\subset S^c. \end{eqnarray*}

In the three cases above, we have shown that $x$ belongs to a closed $\frac{r}{2}$-ball included in $S^c$. The proof of Theorem \ref{th2} is terminated. \hfill $\Box$

In the following example, we prove that the extended exterior $r$-sphere condition assumption of Theorem \ref{th2} cannot be replaced by the exterior $r$-sphere condition.
\begin{example}\label{ex3} In Figure \ref{Fig1}, $S$ is the set that consists of both curves and the dashed region between them. The set $S$ satisfies the exterior $r$-sphere condition, but it does not satisfy the extended exterior $r$-sphere condition since some of the boundary points on the right-hand side have two normal vectors where only one of them is realized by an $r$-sphere. Clearly, $S^c$ is not the union of $r'$-balls for any $r'>0$.
\end{example}
\begin{figure}[h!]
\centering
\includegraphics[scale=1.6]{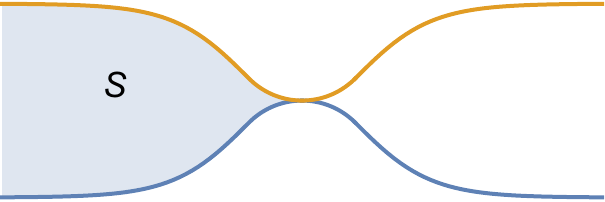}
\caption{\label{Fig1} Example \ref{ex3}}
\end{figure}

\begin{remark} From Theorem \ref{th2}, we deduce that if $S$ is regular closed and satisfies the exterior $r$-sphere condition for some $r>0$, then $S^c$ is the union of closed $\frac{r}{2}$-balls, which is stronger than the \enquote{In addition} part of Theorem \ref{th1}. Of course we can adjust the proof of the \enquote{In addition} part of Theorem \ref{th1} and obtain the radius $\frac{r}{2}$ instead of $r'<\frac{r}{2}$. We did not do this to not complicate the proof of Theorem \ref{th1}, and to show how its proof, when $r'<\frac{r}{2}$, is a {\it direct} consequence of the union of closed balls conjecture \cite[Conjecture 1.2]{nst5}. Note that since $\frac{r}{2}< \frac{nr}{2\sqrt{n^2-1}}$ for all $n\geq2$, the same techniques used in the proof of Theorem \ref{th1} combined with the strong version of the union of closed balls conjecture \cite[Conjecture 1.3]{nst5}, gives that $S^c$ is the union of closed $r'$-balls for any $r'<\frac{nr}{2\sqrt{n^2-1}}$, including $r'=\frac{r}{2}$. Unfortunately, this latter result cannot be confirmed since the strong version of the union of closed balls conjecture remains till now an open question. \end{remark}

\begin{remark} \label{newlastplease} The radius $\frac{r}{2}$ of Theorem \ref{th2} is the {\it largest} radius that works in {\it all} the spaces $\R^n$ for {\it any} $n\geq 2$. Indeed, assume that the radius $\frac{r}{2}$ of Theorem \ref{th2} is replaced by an $r'>0$. We claim that $r'\leq\frac{r}{2}$. Indeed, inspired by \cite[Example 2]{nst5bis}, we consider in $\R^n$ the set $S$ to be the complement of the $(n+1)$ open $r$-balls centered at $C_i$, $i=0,\dots,n$, where for $(e_i)_{i=1}^{n}$ the standard basis of $\R^n$ and $e_0:=0$, we have
\begin{eqnarray*}C_i\bp\bp\bp\bbp&:=&\bp\bp\bp\frac{-nr}{\sqrt{n(n-1)}}\left( \sqrt{\frac{i}{i+1}}\,\sp\sp e_i+\sum_{k=i+1}^{n}\frac{-1}{\sqrt{k(k+1)}}\sp\sp e_k\right)\\&=&\bp\bp\bp \frac{-nr}{\sqrt{n(n-1)}}\bbp\Bigg(\bp
0,\dots,0,\frac{\sqrt{i}}{\sqrt{i+1}}, \frac{-1}{\sqrt{(i+1)(i+2)}},
\frac{-1}{\sqrt{(i+2)(i+3)}},\dots, \frac{-1}{\sqrt{n(n+1)}}\bbp\Bigg)\bbp.\end{eqnarray*}
The set $S$ satisfies the extended exterior $r$-sphere condition (in fact, it is regular closed and satisfies the exterior $r$-sphere condition), but the largest closed balls contained in $S^c$ and containing the origin $0\in S^c$ is of radius $\frac{nr}{2\sqrt{n^2-1}}$, see
Figure \ref{newfig} for $n=2$ and $n=3$. Hence, $r'\leq \frac{nr}{2\sqrt{n^2-1}}$ for all $n\geq 2$. Taking $n \f \infty$ in this latter, we conclude that $r'\leq \frac{r}{2}$.
\begin{figure}
\centering
\includegraphics[width=60mm]{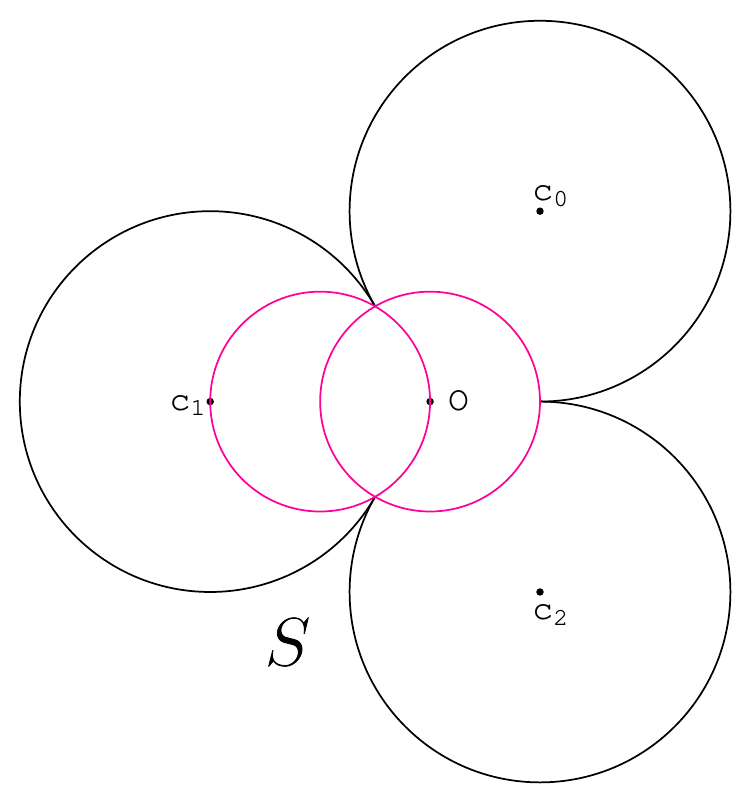}\hspace{1.3cm}
\includegraphics[width=60mm]{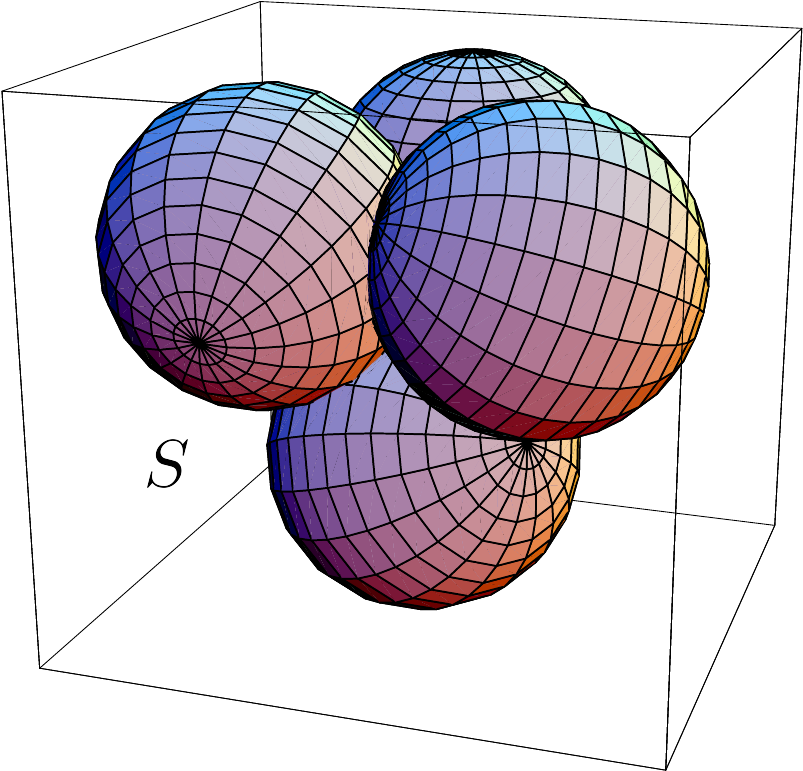}
\caption{\label{newfig} Example of Remark \ref{newlastplease}}
\vspace{0.2cm}
\end{figure}
\end{remark}

\subsection{Proof of Theorem \ref{th3}} Let $S\subset\R^n$ be a closed set such that $\clo(\inte S)$ is epi-Lipschitz and has compact boundary. Since the first implication is straightforward (take $r'=r$), we focus on the second implication. So, we assume that $S$ satisfies the extended exterior $r'$-sphere condition for some $r'>0$. Then $S$ satisfies the exterior $r'$-sphere condition, and hence by Proposition \ref{prop2}$(iv)$, $\clo(\inte S)$ satisfies the exterior $r'$-sphere condition. Now using \cite[Corollary 3.12]{NT1}, applied to $\clo(\inte S)$, we get that $\clo(\inte S)$ is $\rho$-prox-regular for some $\rho>0$. 
\begin{nclaim}{1}\label{claim4} $d(\bdry(\clo(\inte S)),(\bdry S)\setminus \bdry(\inte S))=r_S>0$.
\end{nclaim}
If not, then there exist two sequences $x_n$ and $y_n$ such that $$x_n\in \bdry(\clo(\inte S)),\;\;y_n\in (\bdry S)\setminus \bdry(\inte S)\;\;\hbox{and}\;\;\|x_n-y_n\|\leq\frac{1}{n},\;\;\forall n\geq 1.$$
Since $\clo(\inte S)$ has compact boundary, we get the existence of $\bar{x}\in \bdry(\inte S)$ and a subsequence of $y_n$, we do not relabel, such that $y_n\f\x$. Since $y_n\not\in \clo(\inte S)$, we have from Proposition \ref{prop2}$(iii)$ that $S$ is not epi-Lipschitz at $y_n$. This gives using Proposition \ref{prop2}$(ii)$ the existence of unit vector $\zeta_n$ such that \begin{equation}\label{last2} \zeta_n\in N_S(y_n)\cap -N_S(y_n). \end{equation} Now since $y_n\in (\bdry S)\setminus \bdry(\inte S)$ and $S$ satisfies the extended exterior $r'$-sphere condition, we have that $S$ is $r'$-prox-regular near $y_n$, that is, in a $(\bdry S)$-neighborhood of $y_n$. Hence, Proposition \ref{prop2}$(i)$ and \eqref{last2} yield that $$\zeta_n\in N_S^P(y_n)\cap -N_S^P(y_n).$$
Using the proximal normal inequality and the fact that $\zeta_n$ is unit and realized by an $r'$-sphere, we obtain that $(\zeta_n)_n$ has a subsequence, we do not relabel, that converges to a unit vector $\zeta$ satisfying $$\zeta\in N_S^P(\x)\cap- N_S^P(\x).$$
Since $N_S^P(\x)\subset N_{\clo(\inte S)}^P(\x)$, we deduce that $$\zeta\in N_{\clo(\inte S)}^P(\x)\cap- N_{\clo(\inte S)}^P(\x).$$
This yields using the prox-regularity of $\clo(\inte S)$ and Proposition \ref{prop2}$(i)$, that $$\zeta\in N_{\clo(\inte S)}(\x)\cap- N_{\clo(\inte S)}(\x),$$
which contradicts the fact that $\clo(\inte S)$ is epi-Lipschitz. The proof of the claim is terminated.

Now we prove that $S$ is $r$-prox-regular for $r:=\min\{\rho,r',\frac{r_S}{4}\}$. Let $x\in\bdry S$. If $x\in (\bdry S)\setminus \bdry(\clo(\inte S))$, then from the definition of the extended $r'$-sphere condition, we know that any $0\not=\zeta\in N^P_S(x)$ is realized by an $r'$-sphere. This gives that any $0\not=\zeta\in N^P_S(x)$ is realized by an $r$-sphere. Now we assume that $x\in\bdry(\clo(\inte S)).$ Having $\clo(\inte S)$ $\rho$-prox-regular, it results that for any $0\not=\zeta\in N^P_S(x)\subset N^P_{\clo(\inte S)}(x)$, we have $$B\bp\left(x+\rho\frac{\zeta}{\|\zeta\|};\rho\right)\cap \clo(\inte S)=\emptyset.$$
This gives that for any $0\not=\zeta\in N^P_S(x)$, we have \begin{equation}\label{last1} B\bp\left(x+r\frac{\zeta}{\|\zeta\|};r\right)\cap \clo(\inte S)=\emptyset. \end{equation} 
On the other hand, since $r<\frac{r_S}{4}$ and using Claim \ref{claim4}, we have for any $0\not=\zeta\in N^P_S(x)$ that $$B\bp\left(x+r\frac{\zeta}{\|\zeta\|};r\right)\subset B\left(x;\frac{r_S}{2}\right)\subset [(\bdry S)\setminus \bdry(\inte S)]^c.$$
Combining this latter with Proposition \ref{prop1}$(ii)$ and \eqref{last1}, we get for any $0\not=\zeta\in N^P_S(x)$ that $$B\bp\left(x+r\frac{\zeta}{\|\zeta\|};r\right)\cap S^c=\emptyset.$$
This terminates the proof of Theorem \ref{th3}. \hfill $\Box$

\end{document}